\documentclass[12pt, oneside]{article}
\usepackage[a4paper]{geometry}
\usepackage{amsmath,amsthm,amssymb}
\usepackage{authblk}

\def\lim{\mathop\textrm{lim}}

\def\Exp#1{\raise 1pt \hbox{$e$}^{#1}}
\def\diff{\,\mathrm d}
\def\Cov{\mathrm{Cov}}
\def\Trace{\mathrm{Tr}}

\newtheorem{thm}{Theorem}[section]
\newtheorem{cor}[thm]{Corollary}
\newtheorem{lem}[thm]{Lemma}
\newtheorem{prop}[thm]{Proposition}
\theoremstyle{definition}

\theoremstyle{remark}
\newtheorem{rem}[thm]{Remark}

\numberwithin{equation}{section}

\newcommand{\abs}[1]{\left\vert#1\right\vert}

\newcommand{\fer}[1]{(\ref{#1})}
\newcommand{\R}{\mathbb{R}}
\newcommand{\Mk}{\mathcal{M}(k)}
\newcommand{\Mnk}{\mathcal{M}(n,k)}
\newcommand{\tMnk}{\widetilde{\mathcal{M}}(n,k)}
\newcommand{\Onk}{\mathcal{O}(n,k)}
\newcommand{\On}{\mathcal{O}(n)}
\newcommand{\Ok}{\mathcal{O}(k)}
\newcommand{\Pk}{\mathcal{P}(k)}

\newcommand {\vp} {\varphi}
\newcommand {\lb} {\lambda}
\newcommand {\Chi} {{\bf \raise 1.5pt \hbox{$\chi$}}}

\title{Derivation of a Formula of Blaschke and Petkantschin using Probabilistic Ideas}
\date{\today}
\author{Mohsen Sharifitabar\thanks{sharifitabar@sharif.ir}}
\affil{Dept.\ of Mathematical Sci., Sharif Univ.\ of Tech.}

\begin{document}
\maketitle
\baselineskip1.1\baselineskip
\parindent0pt
\parskip1ex

\begin{abstract}
We give a new proof for the well-known Blaschke--Petkantschin formula 
which is based on the polar decomposition of rectangular matrices
and may be of interest in random matrix theory.
\end{abstract}



%%%%%%%%%%%%%%%%%%%%%%%%%%%%%%%%%%%%%%%%%%%%
\section{Introduction} \label{Sec:intro}

The Blaschke--Petkantschin formula is an integration formula which was introduced by Blaschke\cite{refBlash} and Petkantschin\cite{refPetkan}. Since then, this formula and its generalizations was developed as a powerful tool in many fields, e.g. Analysis, Stereology, Stochastic Geometry \cite{alishahi_sharifitabar_2008, baddeley2004stereology, calka2019poisson, kabluchko2019expected, bonnet2017monotonicity, gotze2019random}. In its classical form, it can be interpreted as a decomposition of $k$-fold product measure of $n$-dimensional Euclidean space. However it has been restated and generalized by many authors; see \cite{moller1985simple, rubin2018blaschke, Forrester2017MatrixPD, moghadasi_2012}. Most of these works have used differential forms. In this paper we present an elegant proof for Blaschke--Petkantschin formula (in matrix form) by a different approach, deploying random Gaussian matrices and their properties.

%%%%%%%%%%%%%%%%%%%%%%%%%%%%%%%%%%%%%%%%%%%%
\section{Matrix Polar Integration Formula and Its Proof} \label{Sec:polin}

Using polar coordinates, one can reduce an integral of a
radial function over $n$-dimensional space to a one-dimensional
integral. The appropriate generalization in this context is to
replace radial with the property that the values of the function
depend only on the relative positions of arguments. 

Before proceeding further, we present some notations which are used throughout this paper.

%%%%%%%%%%%%%%%%%%%%%%%%%%%%%%%%%%%%%%%%%%%%
\paragraph{Some Notations} Let $n,k\in\mathbb{N}$ , $(n\geq k)$.

\begin{itemize}
\item $\Mnk$ : All $n\times k$ real matrices 
\item $\tMnk$ : $\{X\in\Mnk|\ \mathrm{rank}(X)=k\}$ \\
\indent ($\tMnk$ is open and dense in $\Mnk$ if the latter considered as $\R^{nk}$) 
\item $\Onk$ (Orthogonal $n\times k$ matrices) : $\{O\in\Mnk|\ O^T O=1_k\}$ 
\item $\Pk$ (Positive-definite matrices) : $\{P\in\Mk|\ P>0\}$
\end{itemize}

The following theorem makes this generalization precise.

\begin{thm}[Polar Integration Formula (PIF)] \label{thm:pif}
Let $n,k\in\mathbb{N}$, $n\geq k$ and $\vp:
(\R^n)^k\longrightarrow\R$ be orthogonally-symmetric i.e. for any
$O\in O(n)$, $\vp(Ox_1,\ldots ,Ox_k)=\vp(x_1,\ldots ,x_k)$. Then:
\begin{align*}\int_{(\R^n)^k}\vp(x_1,\ldots ,x_k)&\diff x_1\cdots \diff x_k
\\&=\int_{(\R^k)^k}\vp(x_1,\ldots ,x_k)C_{n,k}|\det[x_1,\ldots
,x_k]|^{n-k}\diff x_1\cdots \diff x_k\end{align*} where in
$\vp(x_1,\ldots ,x_k)$ of the right-hand side, $x_1,\ldots ,x_k\in\R^k$ are considered as
vectors in $\R^n$ via the natural embedding of $\R^k$ in $\R^n$,
$[x_1,\ldots ,x_k]$ is the $k\times k$ matrix whose columns are
$x_1,\ldots ,x_k$ and $C_{n,k}$ is a constant.
\end{thm}

In this paper, we prove theorem \fer{thm:pif} using random
matrix point of view. In fact, what we are about to prove is a stronger theorem
in matrix form which is of its own interest.

\begin{prop}[Polar Decomposition] \label{prop:poldec} A matrix $X\in \tMnk$ can be decomposed
uniquely as $X=OP$ where
$O\in\Onk$ and $P\in\Pk$.
\end{prop}

\begin{proof} If $P$ and $O$ are as above, then:
\[X^TX=P^TO^TOP=P^2\quad{=\hspace*{-0.6mm}\Rightarrow}\quad P=\sqrt{X^TX}.\]
\[O=XP^{-1}=X(\sqrt{X^TX})^{-1}\quad(uniqueness)\]
It can be easily checked that these $P$ and $O$ satisfy the desired
conditions (\emph{existence}).
\end{proof}

\paragraph{Two Observations:}
\begin{enumerate}
\item[(i)] To any $x_1,\ldots,x_k\in\mathbb{R}^k$, assign
$X=[x_1,\ldots,x_k]\in\Mnk$ so $\vp:(\R^n)^k\longrightarrow\R$ can
be considered as a real-valued function on $\Mnk$. Now the
orthogonal symmetry of $\vp$ reads $\vp(X)=\vp(UX)$ for all
$U\in\On$ or equivalently if $X=OP$ is the polar decomposition of
$X$, $\vp$ depends only on the positive-definite part of $X$, i.e.
$P$.\\
\item[(ii)] By Polar Decomposition Theorem,
$\tMnk\approx\Onk\times\Pk$. Therefore if one fixes appropriate measures
$\diff X$, $\diff O$ and $\diff P$ on $\Mnk$, $\Onk$ and $\Pk$ respectively, any
integral on $\Mnk$ can be written as an integral over
$\Onk\times\Pk$ after multiplying by the appropriate Jacobian
factor.
\end{enumerate}

\begin{prop} \label{prop:1} Let $n\geq k$. Then $\On$ acts on $\Onk$ (by multiplication on the left).
There exist a unique probability measure $\mu^*$ on $\Onk$
invariant under this action. Moreover, $\mu^*$ is also invariant
under the action of $\Ok$ (by multiplication on the right).
\end{prop}
\begin{proof}
 Both $G=\On$ and $H=\Ok$ are compact Lie groups and
so possess unique Haar probability measures $\mu_G$ and $\mu_H$.
Now $G$ and $H$ act on $X=\Onk$ from left and right, respectively. The
action of $G$ is transitive and $g.(x.h)=(g.x).h$ for any $g\in G$
and $h\in H$ and $x\in X$. Now let $\mu$ be any arbitrary
probability (Borel) measure on X and for any $A\subseteq X$
define:
\[\mu^*(A)=\int_G\int_H\mu(g.A.h)\diff\mu_H(h)\diff\mu_G(g).\]
It is obvious that $\mu^*$ is invariant under the action of $G$ and
$H$. Conversely, for any $G$-invariant probability measure
$\tilde\mu$ on $X$, one can fix some $x_0\in X$ and define $\mu$ on
$G$ as $\mu(K)=\tilde\mu(K.x_0)$. Therefore $\mu$ is a probability measure
because the action is transitive. It is also invariant under
multiplication of G and hence it should be the unique Haar
probability measure on $G$, i.e. $\mu_G$. This means
$\tilde\mu(K.x_0)=\mu_G(K)$ for any $K\subseteq G$. One can replace
$\tilde\mu$ by $\mu^*$ everywhere to conclude
$\tilde\mu(K.x_0)=\mu_G(K)=\mu^*(K.x_0)$. Uniqueness follows by
using the transitivity of the action one more
time.
\end{proof}

We will refer to $\mu^*$ in the above theorem as the \emph{homogeneous measure} on $\Onk$ and integrate functions on $\Onk$
with respect to this measure.

\begin{thm}[Matrix Polar Integration Formula (MPIF)] \label{th:matpif}
For any function $\vp:\Mnk\longrightarrow\mathbb{R}$,
\[\int_{\Mnk}\vp(X)\diff X = \int_{\Pk}\int_{\Onk}\vp(OP)D_{n,k}(\det P)^{n-k}\;\prod_{i<j}(\lb_i+\lb_j)\diff O\diff P\]
where $\diff X=\mathop\prod_{i,j}\diff X_{ij}$, $\diff P=\mathop\prod_{i\leq
j}\diff P_{ij}$ and $\diff O$ is the homogeneous probability measure on $\Onk$
and $\lb_1\geq\ldots\geq\lb_k>0$ are the eigenvalues of $P$ and
$D_{n,k}$ is a constant.
\end{thm}

\paragraph{Matrix Polar Integration Formula $\Rightarrow$ Polar Integration Formula:}
\begin{proof}
It was noted in observation (i) that the left hand side of PIF can
be written as $\int_{\Mnk}\vp(X) dX$. Now using MPIF and noting
that  $\vp(OP)$ is only a function of $P$ (as mentioned in
observation (i)), we obtain,
\[\int_{(\R^n)^k}\vp(x_1,\ldots ,x_k)\diff x_1\cdots \diff x_k=
\int_{\Pk}\vp(P)D_{n,k}(\det
P)^{n-k}\prod_{i<j}(\lb_i+\lb_j)\diff P.\]
Once again, using
MPIF for $n=k$, substituting $\vp(X)|\det X|^{n-k}$ for $\vp$ and
noting that for a $k\times k$ matrix $X$, $|\det X|=\det P$ where
$X=OP$ is the polar decomposition, one obtains:
\begin{align*}\int_{(\R^k)^k}\vp(x_1,\ldots,x_k)&\abs{\det[x_1,\ldots ,x_k]}^{n-k}\diff x_1\cdots \diff x_k \\&= \int_{\Pk}\vp(P)D_{n,k}(\det P)^{n-k}\prod_{i<j}(\lb_i+\lb_j)\diff P.\end{align*}
Comparing these two equalities completes the proof.
\end{proof}

\begin{rem}
It is clear from the above proof that
$C_{n,k}=\frac{D_{n,k}}{D_{k,k}}$.  Our proof does not evaluate
$C_{n,k}$, but using some results from random determinants it can
be computed for even $n-k$.  Let
$\varphi(x_1,\ldots,x_k)=e^{-\frac{1}{2}\sum_{i=1}^k ||x_i||^2}$
which is orthogonally invariant.  In PIF the l.h.s. is a Gaussian
integral which can be easily evaluated.  The r.h.s. is the
$(n-k)$-{th} moment of a Gaussian determinant after
multiplication by an appropriate factor. Thus
\begin{eqnarray*}
(\sqrt{2\pi})^{kn}&=&\int_{(\mathbb{R}^n)^k}e^{-\frac{1}{2}\sum_{i=1}^k
||x_i||^2}\diff x_1\ldots \diff x_k\\
&=&\int_{(\mathbb{R}^k)^k}e^{-\frac{1}{2}\sum_{i=1}^k
||x_i||^2} C_{n,k} |\det [x_1,\ldots,x_k]|^{n-k}\diff x_1\ldots \diff x_k\\
&=&(\sqrt{2\pi})^{k^2}C_{n,k}\mathbb{E}[|\Delta_k|^{n-k}],
\end{eqnarray*}
where $\Delta_k$ is the determinant of a $k\times k$ matrix with
independent standard Gaussian entries. These determinants have
been studied since 1920's when Wishart introduced random
determinants in statistics.  It is known that (see \cite{})
\[
\mathbb{E}[|\Delta_k|^{2r}]= \big(\frac{k}{2}\big)^{-kr}
\prod_{j=1}^k \frac{\Gamma(r+\frac{j}{2})}{\Gamma(\frac{j}{2})}.
\]
For even $n-k$, we obtain:
\[
C_{n,k}=(\pi k)^{\frac{k(n-k)}{2}} \prod_{j=1}^k
\frac{\Gamma(\frac{n-k+j}{2})}{\Gamma(\frac{j}{2})}.
\]
\end{rem}

\begin{prop} \label{prop:2} Let $X$ be an $n\times k$ matrix with independent standard Gaussian
entries and let $X=OP$ be its polar decomposition. Then:
\begin{enumerate}
\item[(i)] $O$ and $P$ are independent.
\item[(ii)] $O$ is distributed according to the homogeneous measure on $\Onk$.
\item[(iii)] $P$ is distributed as a measure on $\Pk$ which is invariant
under orthogonal changes of coordinates $P\longrightarrow V^TPV$
$(V\in\Ok)$.
\end{enumerate}
\end{prop}

To prove this proposition, we need the following lemma.

\begin{lem} \label{lm:1}
Let $U\in\On$, $V\in\Ok$ and $X$ be an $n\times k$ random matrix as
in $\fer{prop:2}$. (i.e. $X_{ij}$'s are independent standard
Gaussian random variables). Then $UXV\sim X$.
\end{lem}

\begin{proof} The entries of $UXV$ are linear
combinations of $X_{ij}$'s and hence jointly Gaussian. So it is
sufficient to compute covariances:
\begin{align*}
\Cov\left((UXV)_{ij}\;,\;(UXV)_{rs}\right)&=
\Cov(\sum_{a,b}U_{ia}X_{ab}V_{bj}\;,\;\sum_{c,d}U_{rc}X_{cd}V_{ds})\\
&=\sum_{a,b,c,d}U_{ia}U_{rc}V_{bj}V_{ds}\Cov(X_{ab},X_{cd})\\
&=\sum_{a,b,c,d}U_{ia}U_{rc}V_{bj}V_{ds}\delta_{ac}\delta_{bd}\\
&=\sum_{a,b}U_{ia}U_{ra}V_{bj}V_{bs}\\
&=(\sum_aU_{ia}U_{ra})(\sum_bV_{bj}V_{bs})\\
&=(UU^T)_{ir}(VV^T)_{js}\;=\;\delta_{ir}\delta_{js}
\end{align*}
\end{proof}

\begin{proof}[Proof of Proposition$\fer{prop:2}$:]
\
\begin{enumerate}
\item[(ii)] By
lemma$\fer{lm:1}$, $UX\sim X$ for any $U\in\On$. But if $X=OP$ is
the polar decomposition of $X$, then $UX=(UO)P$ will be the polar
decomposition of $UX$, so $UO\sim O$ and this is the case for any
$U\in\On$, i.e. the distribution of $O$ is invariant under
the action of $\On$. Now the claim is concluded from Theorem $\fer{prop:1}$.
\item[(i)] Let $\mu(.|P)$ be the conditional probability measure induced
on $\Onk$ knowing the positive-definite part of polar
decomposition to be $P$. Again since $UX\sim X$ and $UX=(UO)P$,
one has $UO\sim O$ under $\mu(.|P)$ which implies that $\mu(.|P)$
is distributed as the homogeneous measure on $\Onk$
which does not depend on $P$. Hence $O$ and $P$ are independent.
\item[(iii)] By Lemma $\fer{lm:1}$, $XV\sim X$ for $V\in\Ok$. Now if
$X=OP$ is the polar decomposition of $X$, the one of $XV$ will be
$XV=(OV)(V^TPV)$ and hence $V^TPV\sim P$.
\end{enumerate}
\end{proof}

\begin{cor} \label{crl:1}
Using the polar decomposition isomorphism $\tMnk\approx\Onk\times\Pk$,
one has $\diff X=\diff O\times \diff\mu_{n,k}$ where $\diff X$ is the standard
Lesbegue measure on $\Mnk$, $\diff O$ is the homogeneous measure on
$\Onk$ and $\mu_{n,k}$ is a measure on $\Pk$ invariant under
orthogonal changes of basis $P\longrightarrow V^TPV$, $V\in\Ok$.
\end{cor}

\begin{proof} Note that the random matrix $X$ described
in Theorem $\fer{prop:2}$ defines the following probability
measure on $\Mnk$:
\[\prod_{i,j}(\frac{1}{\sqrt{2\pi}}e^{-\frac{1}{2}X_{ij}^2}\diff X_{ij})=(\frac{1}{\sqrt{2\pi}})^{nk}e^{-\frac{1}{2}\sum_{i,j}X_{ij}^2}\diff X=(\frac{1}{\sqrt{2\pi}})^{nk}e^{-\frac{1}{2}\Trace (X^TX)}\diff X.\]
Proposition $\fer{prop:2}$ implies
\[(\frac{1}{\sqrt{2\pi}})^{nk}e^{-\frac{1}{2}\Trace (X^TX)}\diff X = \diff O\times\diff \rho_{n,k},\]
where $\rho_{n,k}$ is invariant under $P\longrightarrow V^TPV$. Now
note that $X^TX=P^2$, so
\[\diff X = \diff O\times(\sqrt{2\pi})^{nk}e^{\frac{1}{2}\Trace (P^2)}\diff\rho_{n,k}.\]
Since $\Trace (P^2)=\Trace ((V^TPV)^2)$, defining $\mu_{n,k}$
by:
\[\diff\mu_{n,k} = (\sqrt{2\pi})^{nk}e^{\frac{1}{2}\Trace (P^2)}\diff\rho_{n,k}\]
proves the claim.
\end{proof}

\begin{lem} \label{lm:2}
The measure $\diff P=\prod_{i\leq j}\diff P_{ij}$ on $\Pk$ is invariant under
orthogonal changes of coordinates $P\longrightarrow V^TPV$
$(V\in\Ok)$.
\end{lem}

\begin{proof} We prove the statement on the larger space
of symmetric $k\times k$ matrices; say $\mathcal{S}(k)$. To do so,
Let $P_{ij}$'s $(i\leq j)$ be independent standard Gaussian and
$P_{ji}=P_{ij}$ for $i<j$. Now $Q=V^TPV$ is again a symmetric
random matrix whose entries are linear combinations of
$P$ entries; hence jointly Gaussian. Moreover, for $i\leq j$, $i'\leq j'$
we may write:
\begin{align*}
\Cov(Q_{ij}\;,\;Q_{i'j'})
&= \Cov(\sum_{s,t}V_{si}V_{tj}P_{st}\;,\;\sum_{s',t'}V_{s'i'}V_{t'j'd}P_{s't'})\\
&=\sum_{s,t,s',t'}V_{si}V_{tj}V_{s'i'}V_{t'j'}\Cov(P_{st},P_{s't'})\\
&=\sum_{s,t} V_{si}V_{tj}V_{si'}V_{tj'}\\
&=(\sum_sV_{si}V_{si'})(\sum_tV_{tj}V_{tj'})\\
&=(V^TV)_{ii'}(V^TV)_{jj'}\;=\;\delta_{ii'}\delta_{jj'}\;=\;\delta_{(i,j)(i',j')}
\end{align*}
and therefore $Q=V^TPV\sim P$. Bearing in mind that the probability measure
for the random matrix above is
\[\prod_{i\leq j}(\frac{1}{\sqrt{2\pi}}e^{-\frac{1}{2}P_{ij}^2}\diff P_{ij})
=(\frac{1}{\sqrt{2\pi}})^{\frac{k(k+1)}{2}}e^{-\frac{1}{2}\Trace (P^2)}\diff P,\] we have proved that
\[(\frac{1}{\sqrt{2\pi}})^{\frac{k(k+1)}{2}}e^{-\frac{1}{2}\Trace (P^2)}\diff P
=(\frac{1}{\sqrt{2\pi}})^{\frac{k(k+1)}{2}}e^{-\frac{1}{2}\Trace (Q^2)}\diff Q.\]
But $\Trace (Q^2)=\Trace (P^2)$ and hence $\diff P=\diff Q$.
\end{proof}

\begin{proof}[Proof of MPIF]
 Consider $\tMnk$ as an
open dense subset of $\R^{nk}$ parameterized by $X_{ij}$'s; $1\leq
i\leq n$ , $1\leq j\leq k$ and $\Pk$ as an open set in
$\R^{\frac{k(k+1)}{2}}$ parameterized by $p_{ij}$'s; $1\leq i\leq
j\leq k$ and assume that $u_1,\ldots u_m$ give a smooth
parameterization of an open dense subset of the manifold $\Onk$ in
which $m=nk-\frac{k(k+1)}{2}$. We are going to compute the
Jacobian of the transformation
\[X=F(O,P)=F(u_r,p_{ij};\;1\leq r\leq m,\;1\leq i\leq j\leq k)=OP\]
Let us show the $i^{th}$ column of any matrix $A$ by $A_i$.
Differentiating $X_i=OP_i$ with respect to $u_r$'s and $p_{st}$'s
leads us to the followings:
\begin{eqnarray}
\frac{\partial{X_i}}{\partial{u_r}}=\frac{\partial{O}}{\partial{u_r}}P_i\\[5pt]
\frac{\partial{X_i}}{\partial{p_{st}}}=O\frac{\partial{P_i}}{\partial{p_{st}}}=\delta_{it}O_s+\delta_{is}O_t\\[5pt]
\frac{\partial{X_i}}{\partial{p_{ll}}}=O\frac{\partial{P_i}}{\partial{p_{ll}}}=\delta_{il}O_l
\end{eqnarray}
Now let $\widetilde X$ be the $1\times nk$ row-vector
$[X_1^T,\ldots,X_k^T]$ then the Jacobian is:
\[J\quad=\left[
\begin{array} {c}
\underline{\;\;\;\frac{\partial{\widetilde X_i}}{\partial{u_r}}\;\;\;} \\ \\
\frac{\partial{\widetilde X_i}}{\partial{p_{st}}} \\ \\
\frac{\partial{\widetilde X_i}}{\partial{p_{ll}}}
\end{array}
\right]
\begin{array} {l}
^{(1\leq r\leq m)} \\ \\
_{(1\leq s<t\leq k)} \\ \\
_{(1\leq l\leq k)}
\end{array}
\quad=\left[
\begin{array} {c}
\underline{(\frac{\partial{O}}{\partial{u_r}}P_1)^T\;\cdots\;(\frac{\partial{O}}{\partial{u_r}}P_k)^T} \\ \\
\cdots\;O_t^T\;\cdots\;O_s^T\;\cdots \\ \\
\cdots\;O_l^T\;\cdots
\end{array}
\right]
\begin{array} {l}
\leftarrow u_r \\ \\
\leftarrow p_{st} \\ \\
\leftarrow p_{ll}
\end{array}
\]
where below the separation line, dots mean zeros. Now
$\diff X=|\det(J)|\diff U\times \diff P$. But by Corollary $\fer{crl:1}$, we
also have $\diff X=\diff O\times \diff\mu_{n,k}$. Comparing these two
representation, we conclude that $|\det(J)|$ can be decomposed as
$|\det(J)|=g(O)f(P)$; $\diff O=g(O)\diff U$ and $\diff\mu_{n,k}=f(P)\diff P$.
Moreover, $f(P)$ is invariant under orthogonal change of
coordinates since both $\mu_{n,k}$ and $\diff P$ have this
property. (Corollary $\fer{crl:1}$ and lemma$\fer{lm:2}$). This
shows that $f$ depends only on eigenvalues of $P$; say
$\lb_1,\ldots,\lb_k$. i.e. $f(P)=f(\lb_1,\ldots,\lb_k)$. This fact
allows us to compute $f$ in the case that
$P=\mathrm{diag}(\lb_1,\ldots,\lb_k)$ is a diagonal matrix. Doing so, we
have:
\[\det(J)=\det\left[
\begin{array} {c}
\underline{\lb_1(\frac{\partial{O_1}}{\partial{u_r}})^T\;\;\cdots\;\;\lb_k(\frac{\partial{O_k}}{\partial{u_r}})^T} \\ \\
\cdots\;O_t^T\;\cdots\;O_s^T\;\cdots \\ \\
\cdots\;O_l^T\;\cdots
\end{array}
\right]
\begin{array} {l}
\leftarrow u_r \\ \\
\leftarrow p_{st} \\ \\
\leftarrow p_{ll}
\end{array}\]
Factor $\lb_1$ from first n columns, $\lb_2$ from second n columns
and so on, then we have:
\[\det(J)=(\lb_1\ldots\lb_k)^n\ \det\left[
\begin{array} {c}
\underline{\quad(\frac{\partial{O_1}}{\partial{u_r}})^T\;\;\;\cdots\;\;\;(\frac{\partial{O_k}}{\partial{u_r}})^T\quad} \\ \\
\cdots\;\lb_s^{-1}O_t^T\;\cdots\;\lb_t^{-1}O_s^T\;\cdots \\ \\
\cdots\;\lb_l^{-1}O_l^T\;\cdots
\end{array}
\right]
\begin{array} {l}
\leftarrow u_r \\ \\
\leftarrow p_{st} \\ \\
\leftarrow p_{ll}
\end{array}\]
Again by factoring $(\lb_s\lb_t)^{-1}$ from the rows which
correspond to $p_{st}$'s and $\lb_l^{-1}$ from the rows which
correspond to $p_{ll}$'s we end up with the following:
\[\det(J)=(\lb_1\ldots\lb_k)^{n-k}\ \det\left[
\begin{array} {c}
\underline{(\;\frac{\partial{O_1}}{\partial{u_r}})^T\;\;\;\cdots\;\;\;(\frac{\partial{O_k}}{\partial{u_r}})^T\;} \\ \\
\cdots\;\lb_tO_t^T\;\cdots\;\lb_s O_s^T\;\cdots \\ \\
\cdots\;O_l^T\;\cdots
\end{array}
\right]
\begin{array} {l}
\leftarrow u_r \\ \\
\leftarrow p_{st} \\ \\
\leftarrow p_{ll}
\end{array}\]
Let \[B = \left[
\begin{array} {c}
\underline{\;(\frac{\partial{O_1}}{\partial{u_r}})^T\;\;\;\cdots\;\;\;(\frac{\partial{O_k}}{\partial{u_r}})^T\;} \\ \\
\cdots\;\lb_tO_t^T\;\cdots\;\lb_s O_s^T\;\cdots \\ \\
\cdots\;O_l^T\;\cdots
\end{array}
\right]
\begin{array} {l}
\leftarrow u_r \\ \\
\leftarrow p_{st} \\ \\
\leftarrow p_{ll}
\end{array}\]
Now look at the following $nk\times1$ column vectors:
\[R_l=\left[\begin{array}{c} \vdots\\ \vdots\\ O_l\\ \vdots\\ \vdots \end{array}\right]
\begin{array}{l} \;\\ \;\\ \leftarrow l^{th}\;n\times1\;vector\\ \;\\ \;\end{array}
\quad\quad\quad\quad\quad
R_{st}=\frac{1}{\sqrt2}\left[\begin{array}{c} \vdots\\ O_t\\
\vdots\\ O_s\\ \vdots\end{array}\right]
\begin{array}{l} \;\\ \leftarrow s^{th}\;n\times1\;vector\\ \;\\ \leftarrow t^{th}\;n\times1\;vector\\ \;\end{array}\]
where dots stand for zeros. It can be easily seen that these are
unit orthogonal vectors in $\R^{nk}$ and hence can be extended to an
orthogonal basis and we may form the following orthogonal $nk\times
nk$ matrix:
\[U=\left[\;\cdots\;\left|\begin{array}{c}\\ R_{st} \\ \\ \end{array}\right|\;R_l\;\right]\]
We need some observations here to continue the proof. First,
$R_l$'s and $R_{st}$'s do not depend on $\lb_i$'s; so we may
choose the first $nk-\frac{k(k+1)}{2}$ columns of $U$ independent
of them, too. Next, Observe that the product of $r^{th}$ row of
$B$ by $\sqrt2R_{st}$ yields:
\[\left(\frac{\partial{O_s}}{\partial{u_r}}\right)^T O_t+\left(\frac{\partial{O_t}}{\partial{u_r}}\right)^T O_s\;=\;\left(\frac{\partial{O_s}}{\partial{u_r}}\right)^T O_t+O_s^T\:\frac{\partial{O_t}}{\partial{u_r}}\;=\;\frac{\partial}{\partial{u_r}}\left(O_s^TO_t\right)\;=\;0\]
because $O_s^T O_t=0$. The same is true for $R_l$ since
$O_l^T O_l=1$. Also one can see that the lower rows of $B$ are
orthogonal to $R_l$'s and $R_{st}$'s except in the case that the
indices are the same. All these facts together imply that $BU$ has
the following form:
\[BU=\left[\begin{array}{ccc}\bigstar&0&0\\ \ast&\mathrm{diag}(\frac{\lb_t+\lb_s}{\sqrt2})&0\\ \ast&0&1_k\end{array}\right]\]
in which, $\lb_i$'s do not appear in $\bigstar$. Keeping in mind
that $\det(B)=\det(BU)$, we find out that there exists a constant
$D_{n,k}$ such that
\[f(P)=f(\lb_1,\ldots,\lb_k)=D_{n,k}(\lb_1\ldots\lb_k)^{n-k}\prod_{s<t}(\lb_s+\lb_t)=D_{n,k}(\det P)^{n-k}\prod_{s<t}(\lb_s+\lb_t)\]
which completes the proof.
\end{proof}
%%%%%%%%%%%%%%%%%%%%%%%%%%%%%%%%%%%%%%%%%%%%

\bibliographystyle{abbrv}
\bibliography{1}

\end{document}